\documentclass[12pt]{amsart}

\usepackage{amsmath}
\usepackage{amssymb}
\usepackage{amsfonts}
\usepackage{amsthm}
\usepackage{enumerate}
\usepackage{hyperref}
\usepackage{graphicx, xcolor}
\usepackage{color}
\theoremstyle{plain}
\newtheorem{thm}{Theorem}[section]

\newtheorem{prop}[thm]{Proposition}
\newtheorem{lem}[thm]{Lemma}
\newtheorem{cor}[thm]{Corollary}

\theoremstyle{definition}

\newtheorem{dfns-rems}[thm]{Definitions and Remarks}
\newtheorem{notas-rems}[thm]{Notations and Remarks}
\newtheorem{exmps-rems}[thm]{Examples and Remarks}

\begin{document}

% ------------------------------------------------------------------------

\title[Roman $\{2\}$-domination in graphs \\
and graph products]{Roman $\{2\}$-domination in graphs\\
 and graph products}

% ------------------------------------------------------------------------

\author[F. Alizadeh, H. R. Maimani, L. Parsaei Majd, and M. Rajabi Parsa]{Faezeh Alizadeh}

\address{F. Alizadeh, Mathematics Section, Department of Basic Sciences,
Shahid Rajaee Teacher Training University, P.O. Box 16785-163, Tehran,
Iran.}

\email{alizadeh$_{-}$faezeh@ymail.com}

\author[]{Hamid Reza Maimani}

\address{H. R. Maimani, Mathematics Section, Department of Basic Sciences,
Shahid Rajaee Teacher Training University, P.O. Box 16785-163, Tehran,
Iran.}

\email{maimani@ipm.ir}

\author[]{Leila Parsaei Majd}

\address{L. Parsaei Majd, Mathematics Section, Department of Basic Sciences,
Shahid Rajaee Teacher Training University, P.O. Box 16785-163, Tehran,
Iran.}

\email{leila.parsaei84@yahoo.com}

\author[]{Mina Rajabi Parsa}

\address{M. Rajabi Parsa, Mathematics Section, Department of Basic Sciences,
Shahid Rajaee Teacher Training University, P.O. Box 16785-163, Tehran,
Iran.}

\email{mina.rparsa@gmail.com}

% ------------------------------------------------------------------------

\begin{abstract}
For a graph $G=(V, E)$ of order $n$, a Roman $\{2\}$-dominating function $f: V\rightarrow \{0, 1, 2\}$ has the property that for every vertex $v\in V$ with $f(v)=0$, either $v$ is adjacent to a vertex assigned $2$ under $f$, or $v$ is adjacent to at least two vertices assigned $1$ under $f$. In this paper, we classify all graphs with Roman $\{2\}$-domination number belonging to the set $\{2, 3, 4, n-2, n-1, n\}$. Furthermore, we obtain some results about Roman $\{2\}$-domination number of some graph operations.
\end{abstract}

% ------------------------------------------------------------------------

\subjclass[2010]{Primary: 05C69; Secondary: 05C76}

% ------------------------------------------------------------------------

\keywords{Roman $\{2\}$-domination; Cartesian product; Grid graph. }

% ------------------------------------------------------------------------

%\thanks{The research of H. R. Maimani and S. Yassemi was in part supported
%by a grant from IPM (No. 91050??? and No. 91130214)}

% ------------------------------------------------------------------------

\maketitle

%%%%%%%%%%%%%%%%%%%%%%%%%%%%%%%%%%%%%%%%%%%%%%%%%%%%%%%%%%%%%%%%%%%%%%%%%%
\section{Introduction} \label{sec1}
We study \textit{Roman $\{2\}$-dominating} functions defined in \cite{[1]}. We first present some necessary terminology and notation. Let $G = (V,E)$ be a graph with vertex set $V = V (G)$ and edge set $E=E(G)$. The \textit{open neighborhood} $N(v)$ of a vertex $v$ consists of the vertices adjacent to $v$, and its \textit{closed neighborhood} is $N[v]=N(v)\cup\{v\}$. The degree of $v$ is the cardinality of its open neighborhood. Let $\Delta(G)$ be the maximum degree of the graph $G$. If $S$ is a subset of $V$, then $N(S) = \bigcup_{x\in S}N(x)$, $N[S]=\bigcup_{x\in S}N[x]$, and the subgraph induced by $S$ in $G$ is denoted $G[S]$. 

A \textit{dominating set} of $G$ is a subset $S$ of $V$ such that every vertex in $V-S$ has at least one neighbor in $S$, in other words, $N[S]=V$. The \textit{domination number} $\gamma(G)$ is the minimum cardinality of a dominating set of $G$. By \cite{ref8}, a subset $S \subseteq V$ is a $2$-dominating set if every vertex of $V-S$ has at least two neighbors in $S$. The \textit{$2$-domination number} $\gamma_2(G)$ is the minimum cardinality of a $2$-dominating set of $G$. 

Motivated by Stewart's \cite{ref13} article on defending the Roman Empire, Cockayne et al. introduced Roman dominating functions in \cite{ref6}. For Roman domination, each vertex in the graph model corresponds to a location in the Roman Empire, and for protection, legions (armies) are stationed at various locations. A location is protected by a legion stationed there. A location having no legion can be protected by a legion sent from a neighboring location. However, this presents the problem of leaving a location unprotected (without a legion) when its legion is dispatched to a neighboring location. In order to prevent such problems, Emperor Constantine the Great \cite{ref6} decreed that a legion cannot be sent to a neighboring location if it leaves its original station unprotected. In other words, every location with no legion must be adjacent to a location that has at least two legions. This defense strategy prompted the following definition in \cite{ref6}.

A function $f : V (G) \rightarrow \{0, 1, 2\}$ is a \textit{Roman dominating function} (RDF) on $G$ if every vertex $u \in V$ for which $f(u) = 0$ is adjacent to at least one vertex $v$ for which $f(v) = 2$. The weight of an RDF is the value $f(V (G)) = \sum_{u\in V (G)} f(u)$. The \textit{Roman domination number} $\gamma_R(G)$ is the minimum weight of an RDF on $G$.
A vertex $v$ with $f(v) = 0$ is said to be undefended with respect to $f$ if it is not adjacent to a vertex $w$ with $f(w) > 0$.

In this paper, we study Roman $\{2\}$-dominating functions. These functions are closely related to $\{2\}$-dominating functions introduced in \cite{ref7} as follows. For a graph $G$, a \textit{$\{2\}$-dominating function} is a function $f : V \rightarrow \{0, 1, 2\}$ having the property that for every vertex $u \in V$, $f(N[u]) \geq 2$. The weight of a $\{2\}$-dominating function is the sum $f(V ) = \sum_{v\in V} f(v)$, and the minimum weight of a $\{2\}$-dominating function $f$ is the $\{2\}$-domination number, denoted by $\gamma_{\{2\}}(G)$.

A Roman $\{2\}$-dominating function $f$ relaxes the restriction that for every vertex $u \in V$, $f(N[u]) = \sum_{v\in N[u]} f(v) \geq 2$ to only requiring that this property holds for every vertex assigned $0$ under $f$. Formally, a Roman $\{2\}$-dominating function $f : V \rightarrow \{0, 1, 2\}$ has the property that for every vertex $v \in V$ with $f(v) = 0$, $f(N(u)) \geq 2$, that is, either there is a vertex $u \in N(v)$, with $f(u) = 2$, or at least two vertices $x, y \in N(v)$ with $f(x) = f(y) = 1$. In terms of the Roman Empire, this defense strategy requires that every location with no legion has a neighboring location with two legions, or at least two neighboring locations with one legion each. Note that for a Roman $\{2\}$-dominating function $f$, it is possible that $f(N[v]) = 1$ for some vertex with $f(v) = 1$. The \textit{weight} of a Roman $\{2\}$-dominating function $f$ is defined as 
$$w(f) = f(V) = \sum_{v\in V} f(v),$$
and the minimum weight of a Roman $\{2\}$-dominating function $f$ is the Roman $\{2\}$-domination number, denoted by $\gamma_{\{R2\}}(G)$.

\begin{lem}\cite[Corollary 10]{[1]}
for a cycle $C_{n}$ and a path $P_{n}$ we have
\begin{center}
$ \gamma _{\lbrace R2\rbrace }(C_{n})=\lceil \frac{n}{2}\rceil, ~~~~~~    \gamma _{\lbrace R2\rbrace }(P_{n})=\lceil \frac{n+1}{2}\rceil.$
\end{center}
\end{lem}
\begin{prop}\cite[Proposition 5]{[1]}
For every graph $G$; $\gamma_{\{R2\}}(G)\leqslant \gamma_{2}(G)$.
\end{prop}
%%%%%%%%%%%%%%%%%%%%%%%%%%%%%%%%%%%%%%%%%%%%%%%%%%%%%

For graphs $G$ and $H$, The \textit{join} of graphs $G$ and $H$ is the graph $G\vee H$ with
the vertex set $V=V(G)\cup V(H)$ and the edge set $E=E(G) \cup E(H) \cup \{uv : u\in V(G), v\in V(H)\}$. 

The \textit{Corona} $G[H]$ of $G$ and $H$ is constructed as follows: \\
Choose a labeling of the vertices of $G$ with labels
$1, 2, \ldots, n$. Take one copy of $G$ and $n$ disjoint copies of $H$, labeled $H_1, \ldots, H_n$,
and connect each vertex of $H_i$ to vertex $i$ of $G$.

The \textit{Cartesian product} of two graphs $G$ and $H$, denoted
by $G\square H$, has vertex set
$V(G\square H) = V (G) \times V(H)$, where two distinct vertices $(u, v)$ and $(x, y)$ of $G\square H$ are adjacent if either
$$u=x ~\text{and} ~vy \in E(H) ~\text{or} ~ v = y ~\text{and} ~ux \in E(G).$$

The \textit{grid} graph $G_{m, n}$ is the Cartesian product of $P_m$ and $P_n$.
In 1983, Jacobson and Kinch \cite{[26]} established the exact values of $\gamma(G_{m,n})$ for $2 \leqslant m \leqslant 4$ which are the first results on the domination number of grids. Also, In 1993, Chang and Clark \cite{[6]} found those
of $\gamma(G_{m,n})$ for $m=5$ and $6$. Fischer found those of $\gamma(G_{m,n})$ for $m\leqslant 21$ (see Goncalves et al. \cite{[16]}). Recently, Goncalves et al. \cite{[16]} finished the computation of $\gamma(G_{m,n})$ when $24\leqslant m\leqslant n$. 
In \cite{new}, the authors have obtained the values of $\gamma_{2}(G_{m,n})$ for $2\leqslant m\leqslant 4$.
In this paper, we will give some boundaries for $\gamma_{\{R2\}}(G_{m,n})$ for $2\leqslant m\leqslant 4$.

\section{Graphs with small or large Roman $\{2\}$-domination number} \label{sec2}
In this section we provide a characterization of all connected graphs $G$ of order $n$ with Roman $\{2\}$-domination number belonging to $\{2, 3, 4, n-2, n-1, n\}$. 
Let $f=(V_0, V_1, V_2)$ be a function $f : V \rightarrow \{0, 1, 2\}$ on a graph $G = (V, E)$, where $V_i = \{v \vert f(v)=i \}$ for $i \in \{0, 1, 2\}$.

\begin{prop}\label{2-roman2}
Let $G$ be a graph. $\gamma_{\{R2\}}(G)=2$ if and only if
$G=\overline{K_t}\vee H$ for $t=1, 2$ and for some graph $H$.
\end{prop}
\begin{proof}
Let $f=(V_0, V_1, V_2)$ be a $\gamma_{\{R2\}}(G)$-function with weight $2$. Hence, we have two cases. If there exists a vertex $z$ with $z\in V_2$, then all other vertices of $G$ are adjacent to $z$. Therefore, $G=\overline{K_1} \vee H$ for some graph $H$. If there are two vertices $u$ and $v$ in $V_1$, then all other vertices of $G$ are adjacent to both vertices $u$ and $v$. If $u$ and $v$ are adjacent, then $G=\overline{K_1} \vee H$ for some graph $H$, and if $u$ and $v$ are not adjacent, then $G=\overline{K_2} \vee H$ for some induced subgraph $H$ of $G$. Conversely, it is not hard to see the result.
\end{proof}
For a graph $G$, define $N_{i}(G)$ for $i=1, \ldots, n-1$ as follows,
$$N_{i}(G)=\{v\in V: deg(v)=i\}.$$

\begin{prop}\label{2-roman3}
Let $G$ be a graph. Then $\gamma_{\{R2\}}(G)=3$ if and only if one of the following holds:
\begin{enumerate}
\item[(i)]
$\Delta(G)=n-2$ and $N_{n-2}(G)$ is a clique,

\item[(ii)]
$\Delta(G)< n-2$ and $\gamma_2(G)=3$.
\end{enumerate}
\end{prop}
\begin{proof}
Let $f=(V_0, V_1, V_2)$ be a $\gamma_{\{R2\}}(G)$-function with weight $3$. By Proposition \ref{2-roman2}, $\Delta(G)\leqslant n-2$. At first, suppose that $\Delta(G)=n-2$.  Let $\vert N_{n-2}(G)\vert =1$, $v\in N_{n-2}(G)$ and $u\notin N(v)$. Then, set $v\in V_2$ and $u\in V_1$. Now, if $\vert N_{n-2}(G)\vert \geqslant 2$. Consider two vertices $u$ and $v$ in $N_{n-2}(G)$. If $u$ and $v$ are not adjacent, then $u$ and $v$ are adjacent to all other vertices of $G$, and hence $G=\overline{K_2} \vee H$, which is a contradiction by Proposition \ref{2-roman2}. Thus, $N_{n-2}(G)$ is a clique. \\
If $\Delta(G)< n-2$, then there are three vertices $u, v$ and $w$ in $V_1$. Hence, $\{u, v, w\}$ is a $2$-dominating set, so $\gamma_2(G)\leqslant 3$. Since $\gamma_{\{R2\}}(G)=3$, we have $\gamma_2(G)\geqslant 3$. So, $\gamma_2(G)=3$. The converse proof can be easily checked.
\end{proof}
\begin{prop}\label{dom4}
Let $G$ be a graph. Then $\gamma_{\{R2\}}(G)=4$ if and only if $\Delta(G)\leqslant n-3$ and $\gamma_{2}(G)\geqslant 4$ as well as $G$ satisfies one of the following conditions, 
\begin{enumerate}
\item[(i)]
$\gamma(G)=2$,
\item[(ii)]
$\gamma_{2}(G)=4$,
\item[(iii)]
There exists a vertex $v\in V(G)$ such that $\gamma_{2}(G[V(G)-N[v]])=2$.
\end{enumerate}
\end{prop}
\begin{proof}
Suppose that $\gamma_{\{R2\}}(G)=4$. By Propositions \ref{2-roman2} and \ref{2-roman3}, we have $\Delta(G)\leqslant n-3$ and $\gamma_{2}(G)\geqslant 4$. Let $f=(V_0, V_1, V_2)$ be a $\gamma_{\{R2\}}(G)$-function. We consider three cases. First case, if $\vert V_2 \vert = 2$, then $\gamma(G)=2$. Second case, $\vert V_1 \vert = 4$, so $\gamma_{2}(G)=4$. Finally, $\vert V_1 \vert = 2$ and $\vert V_2 \vert = 1$. Suppose that $V_1=\{u, w\}$ and $V_2=\{v\}$. Obviously, each vertex in $(V(G)-\{u, w\})-N[v]$ must be connected to both $u$ and $w$. Hence, $\gamma_{2}(G[V(G)-N[v]])=2$.
Conversely, the result is obvious if we have $(i)$ or $(ii)$. Now, suppose that $G$ satisfies $(iii)$. Since $\Delta(G)\leqslant n-3$ and $\gamma_{2}(G)\geqslant 4$, by Propositions \ref{2-roman2} and \ref{2-roman3}, $\gamma_{\{R2\}}(G)\geqslant 4$. On the other hand, assume that $\{u, w\}$ is a $2$-dominating set for $G[V(G)-N[v]]$. If we assign a $2$ to $v$ and a $1$ to $u$ and $w$, we can show that $\gamma_{\{R2\}}(G)\leqslant 4$. Thus, $\gamma_{\{R2\}}(G)=4$.
\end{proof}

\begin{cor}
Let $n_1, n_2, \ldots, n_r$ be the positive integers such that $n_1\leqslant n_2\leqslant \cdots \leqslant n_r$. Then Roman $\{2\}$-domination number of the $n_r$-partite graph $K_{n_1, n_2, \ldots, n_r}$ is as follows:

\begin{equation*}
\gamma_{\{R2\}}(K_{n_1, n_2, \ldots, n_r})= \left\{
\begin{array}{rl}
2 & \quad \text{if}~ n_1=1~\text{or}~2, \\
3 & \quad \text{if}~ n_1=3, \\
4 &  \quad \text{otherwise.}
\end{array} \right.
\end{equation*}
\end{cor}

\begin{prop}
Let $G$ be a connected graph with order $n$. The following statements hold. 
\begin{enumerate}
\item[(a)]
$\gamma_{\{R2\}}(G)= n$ if and only if $G=K_n$ for $n=1, 2$.
\item[(b)]
$\gamma_{\{R2\}}(G)= n-1$ if and only if $G$ is a $C_3$, $P_3$ or $P_4$.
\end{enumerate}
\end{prop}
\begin{proof}
For $(a)$ it is clear that $\Delta(G)\leqslant 1$. 
For $(b)$, if $G$ is one of the $C_3$, $P_3$ or $P_4$, then the claim is true. Conversely, assume that $\gamma_{\{R2\}}(G)= n-1$. Obviously $\Delta(G)=2$. Among all $\gamma_{\{R2\}}(G)$-functions, let $f=(V_0, V_1, V_2)$ be one with $\vert V_2\vert$ as small as possible. It is easy to see that $V_2=\emptyset$ and $\vert V_0\vert=1$. Suppose that $v\in V_0$ for some vertex $v\in V(G)$, so $deg(v)=2$. Also, each vertex except $v$ can be adjacent to at most one vertex in $V_1$. Hence, the vertices which have the degree $2$ are at most $v$ and $N(v)$. Therefore, we have just three graphs, $C_3$, $P_3$ or $P_4$.
\end{proof}

Now, we need the following graphs in Proposition \ref{n-2}.
$\hat{E}_6$ is a tree obtained from $K_{1,3}$ by subdividing each edge exactly once.
$D_7$ is also a tree obtained from $K_{1,3}$ by subdividing one edge three times, (see \cite{brou-hamer}).
We define the graph $H_2$ such that it is a graph with a $4$-cycle and a path of order $2$ joined to one of the vertices of the $4$-cycle.

\begin{figure}[!htb]
\minipage{0.90\textwidth}
\includegraphics[width=\linewidth]{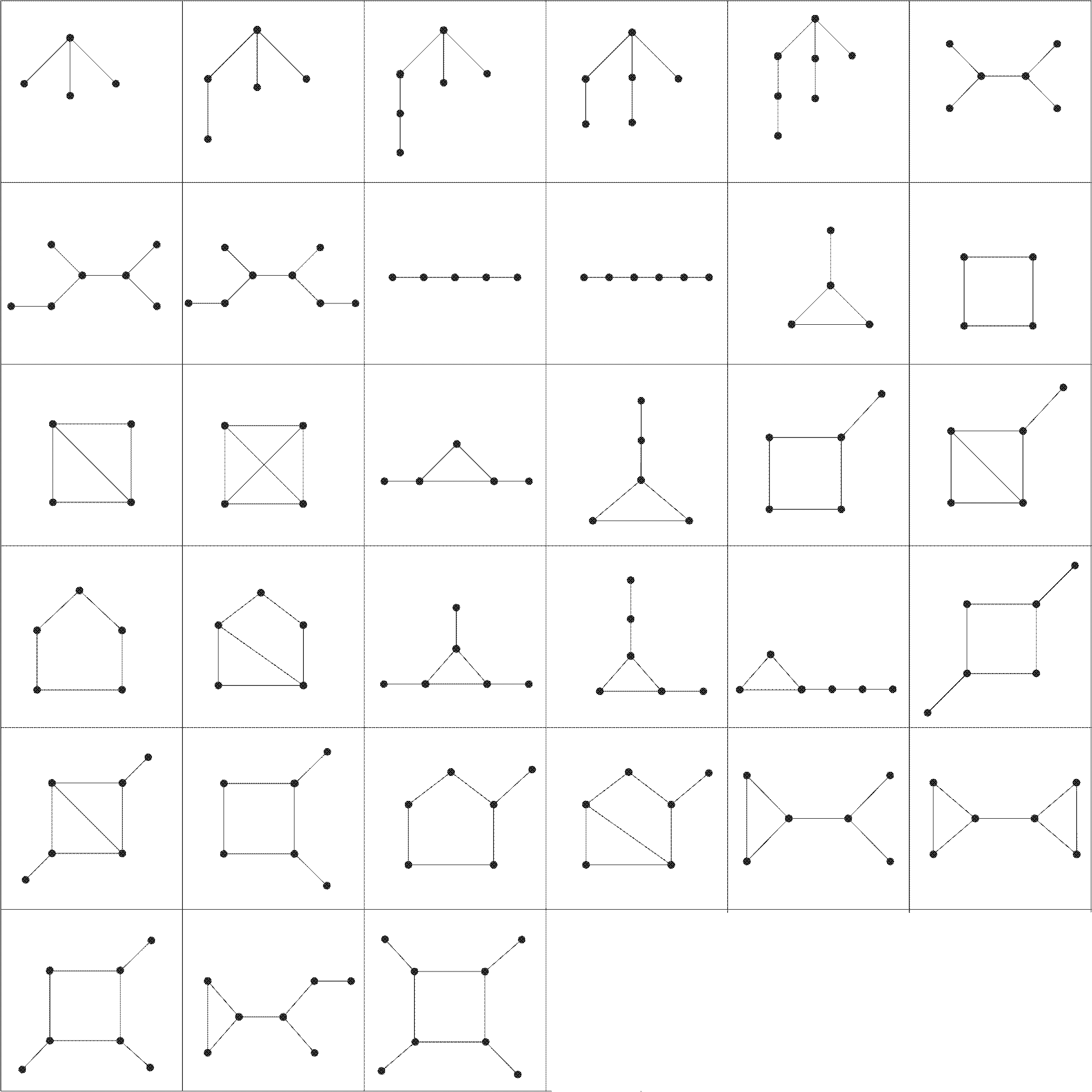}
\caption{}\label{1000}
\endminipage
\end{figure}
\begin{prop}\label{n-2}
Let $G$ be a connected graph with order $n$. Then $\gamma_{\{R2\}}(G)=n-2$ if and only if $G$ is one of the graphs listed in Fig. \ref{1000}.
\end{prop}
\begin{proof}
Suppose that $\gamma_{\{R2\}}(G)=n-2$, then the following conditions hold,
\begin{enumerate}
\item[(i)]
$\Delta(G)\leqslant 3$,
\item[(ii)]
each non-adjacent pair of vertices with degree $3$ has exactly two common neighbours,
\item[(iii)]
$G$ does not have one of the graphs $P_7$, $C_6$, $\hat{E}_6$, $D_7$, and $H_2$ as subgraph.
\end{enumerate}
If there exists a vertex $v\in V(G)$ with degree at least $4$, then $\gamma_{\{R2\}}(G)\leqslant n-3$.
Also, if there exists a pair of nonadjacent vertices with degree $3$ having zero, one or three common neighbours, then we obtain $\gamma_{\{R2\}}(G)\leqslant n-3$. Moreover, Roman $\{2\}$-domination number of each of graphs $P_7$, $C_6$, $\hat{E}_6$, $D_7$, and $H_2$ is $n-3$. Thus, they cannot be as a subgraph of $G$. It is not hard to see that all graphs which have the above three properties are listed in Fig. \ref{1000}.
Conversely, it is easy to verify that for all graphs $G$ listed in Fig. \ref{1000}, we have $\gamma_{\{R2\}}(G)= n-2$.
\end{proof}

%%%%%%%%%%%%%%%%%%%%%%%%%%%%%%%%%%%%%%%%%%%%%%%%%%%%%
\section{Graph products}
In this section we study Roman $\{2\}$-domination on some graph products.\\
Also, in the following theorems we classify Roman $\{2\}$-domination for join of two graphs.
%$k=\mathrm{min}\{ \gamma_{\{R2\}}(G), \gamma_{\{R2\}}(H)\}$
\begin{thm}\label{join2}
Let $G$ and $H$ be two graphs. Then $\gamma_{\{R2\}}(G\vee H)\leqslant 4$. Moreover, if $k=\gamma_{\{R2\}}(G)\leqslant \gamma_{\{R2\}}(H)$, then we have
\begin{enumerate}
\item[(a)]
$k\leqslant 2$ if and only if $\gamma_{\{R2\}}(G\vee H)=2$,

\item[(b)]
$k=3$ or $k=4$ and $\gamma(G)=2$ if and only if $\gamma_{\{R2\}}(G\vee H)=3$.

\end{enumerate}
\end{thm}
\begin{proof}
The first assertion is obvious because for each graph $G$, ${\gamma}_{\{R2\}}(G)\leqslant 2\gamma(G)$.
For $(a)$, assume that $k=1$, then $G=K_{1}$. It is sufficient to use Proposition \ref{2-roman2}.
Now, suppose $k=2$. By Proposition \ref{2-roman2}, $G\vee H=\overline{K_t}\vee F$ for $t=1~\text{or}~2$, and for some graph $F$.
Conversely, let ${\gamma}_{\{R2\}}(G\vee H)=2$. By Proposition \ref{2-roman2}, there exists a graph $L$ such that $G \vee H=\overline{K_t}\vee L$ for $t=1~\text{or}~2$. Without loss of generality, assume that $G$ and $H$ are not clique graphs. One can check that the vertex set of $\overline{K_t}$ for $t=1, 2$ is a subset of $V(G)$ or $V(H)$. If $t=1$ and $V(\overline{K_1})\subseteq V(H)$, then $\gamma_{\{R2\}}(H)\leqslant 2$. Also, if $t=2$ and $V(\overline{K_2})\subseteq V(H)$, then $\gamma_{\{R2\}}(H)\leqslant 2$. Similarly, one can check two other cases.
Hence, $\gamma_{\{R2\}}(G)\leqslant 2$.

For $(b)$, if $k=3$, then $\gamma_{ \lbrace R2 \rbrace}(G \vee H) \leqslant 3$. By $(a)$, $\gamma_{ \lbrace R2 \rbrace}(G \vee H) \geqslant 3$.
Now, let $k=4$ and $\{u, v\}\subseteq V(G)$ be a minimum dominating set for $G$, and $w$ be an arbitrary vertex in $V(H)$. It is seen that $\{u, v, w\}$ is a $2$-dominating set for $G\vee H$. By Proposition \ref{2-roman3} we have ${\gamma}_{\{R2\}}(G\vee H)=3$.
Conversely, let $\gamma_{\lbrace R2 \rbrace}(G \vee H)=3$. By $(a)$, $k\geqslant 3$. 
We consider two cases using Proposition \ref{2-roman3} to complete the proof.  
First assume that $\gamma_{2}(G \vee H)=3$. Let $\{u, v, w\}\subseteq V(G \vee H)$ be a $2$-dominating set on $G\vee H$. Without loss of generality, we consider two subcases,
\begin{enumerate}
\item[(i)]
If $\{u, v, w\}\subseteq V(G)$, then by $(a)$, $\gamma_{\{R2\}}(G)=3$.
\item[(ii)]
If $\{u, v\}\subseteq V(G)$ and $w\in V(H)$, then $\gamma(G)=2$. So by $(a)$, $3\leqslant \gamma_{\{R2\}}(G)\leqslant 4$.
\end{enumerate}

Let $\vert V(G)\vert =n$ and $\vert V(H)\vert =m$. Suppose that $\Delta(G\vee H)=n+m-2$ and $N_{n+m-2}(G\vee H)$ is clique. Let $\vert N_{n+m-2}(G\vee H)\vert = 1$, $v\in N_{n+m-2}(G\vee H)$ and $u\notin N(v)$. Since $\gamma_{\lbrace R2 \rbrace}(G \vee H)=3$, then $\{u, v\}$ is a subset of $V(G)$ or $V(H)$. Without loss of generality, let $\{u, v\}$ be a subset of $V(H)$. Hence, $\gamma_{\lbrace R2 \rbrace}(H)=3$, and by Proposition \ref{2-roman2}, $\gamma_{\lbrace R2 \rbrace}(G)=3$. Now, assume that $\vert N_{n+m-2}(G\vee H)\vert \geqslant 2$. Consider two vertices $u, v\in N_{n+m-2}(G\vee H)$ and a vertex $w\notin N_{n+m-2}(G\vee H)$. By Proposition \ref{2-roman2} and $\gamma_{\lbrace R2 \rbrace}(G\vee H)=3$, we can check that $\{u, v\}\nsubseteq V(G)$ and $\{u, v\}\nsubseteq V(H)$. Finally, we can conclude that $k=3$.
\end{proof}
%%%%%%%%%%%%%%%%%%%%%%%%%%%%%%%%%5
In the following theorem we obtain Roman $\{2\}$-domination number for the Corona
product of two graphs.
 
\begin{thm}\label{corono}
Let $G$ and $H$ be two graphs such that the order of $G$ is $n$. If $H=K_1$, then $\gamma_{\{R2\}}(G[H])=n+\gamma(G)$, otherwise $\gamma_{\{R2\}}(G[H])=2n$.
\end{thm}
\begin{proof}
Let $H=K_1$. Easily we can show that for every graph $G$, $\gamma_{\{R2\}}(G[K_1] )\leqslant n+\gamma(G)$. On the other hand, assume that $f=(V_0, V_1, V_2)$ is a $\gamma_{\{R2\}}(G[K_1])$-function. Without loss of generality, suppose that $nK_1\subseteq V_0 \cup V_1$. Also, let $\ell K_1\in V_1$ and $(n-\ell)K_1\in V_0$. Thus, $V_2 \subseteq V(G)$. Moreover, $(V_1 \cap V(G))\cup V_2$ forms a dominating set for $G$.
\begin{align*}
w(f)&=\ell+\vert V_1 \cap V(G)\vert+2\vert V_2\vert\\
&\geqslant \ell + \gamma(G) + \vert V_2\vert \\
&\geqslant n+\gamma(G).
\end{align*}
For the second assertion, let $V(G)=\{v_1, v_2, \ldots, v_n\}$ and $f$ be a $\gamma_{\{R2\}}(G[H])$-function. Then, 
$$w(f)=w(f\vert_{H_1})+w(f\vert_{H_2})+\ldots +w(f\vert_{H_n})\geqslant 2n,$$
where $H_i=v_i\vee H$ for $i=1, 2, \ldots, n$. So, $\gamma_{\{R2\}}(G[H])=2n$.
\end{proof}

Moreover, we state a bound and some results about Cartesian product of graphs.
Let $G$ and $H$ be two graphs with $V(G)=\{v_1, v_2, \ldots, v_n\}$ and $V(H)=\{u_1, u_2, \ldots, u_m\}$. In $G\square H$, we define $G^i$ and $H^j$  for $i=1, \ldots, m$ and $j=1, \ldots, n$, as $i$th layer and $j$th layer of $G$ and $H$, respectively as follows,
$$G^i=\{(v, u_i): v\in V(G)\}, ~~H^j=\{(v_j, u): u\in V(H)\}.$$
\begin{thm}
$\gamma_{\{R2\}}(G\square H)\leqslant \min \left\{ \gamma_{\{R2\}}(G)|V(H)|, \gamma_{\{R2\}}(H) |V(G)| \right\}$. 
Also, this bound is sharp.
\end{thm}
\begin{proof}
Let $f$ be a $\gamma_{\{R2\}}$-function for $H$. Consider each copy of $H$ with $\gamma_{\{R2\}}$-function $f$ in cartesian product $G\square H$. Since we have $\vert V(G)\vert$ copies of $H$, it is easy to see that $\gamma_{\{R2\}}(G\square H)\leqslant \gamma_{\{R2\}}(H) |V(G)|$. By a similar way, we have $\gamma_{\{R2\}}(G\square H)\leqslant \gamma_{\{R2\}}(G)|V(H)|$.
In order to prove this bound is sharp, consider $\gamma_{\{R2\}}(K_{1,n} \square P_{2})=2\gamma_{\{R2\}}(K_{1,n})=4$, for $n\geqslant 3$, see Proposition \ref{dom4}.
\end{proof}
\begin{thm}
Let $m$ and $n$ be two positive integers with $n
\leqslant m$. Then $$\gamma_{\{R2\}}(K_{n}\square K_{m})= \min \{ m, 2n \}.$$
\end{thm}
\begin{proof}
Let $V(K_n)=\{v_1, v_2, \ldots, v_n\}$ and $V(K_m)=\{u_1, u_2, \ldots, u_m\}$. 
Suppose that $\gamma_{\{R2\}}(K_{n}\square K_{m}) < \min \{ m, 2n \}$, and let $f=(V_0, V_1, V_2)$ be a $\gamma_{\{R2\}}(K_{n}\square K_{m})$-function. Thus, we can say that there exists the layer $K_{n}^{i}$ for some $1\leqslant i\leqslant m$, such that $w(f\vert_{K_{n}^{i}})=0$. On the other hand, we can find a layer $K_{m}^{j}$ for some $1\leqslant j\leqslant n$, with $w(f\vert_{K_{m}^{j}})\leqslant 1$. It is easy to see that $(v_i, u_j)\in V_0$ and $f(N(v_i, u_j))\leqslant 1$. Therefore, we achieve a contradict.
Now to get the equality, consider a Roman $\{2\}$-dominating function on $K_{n}\square K_{m}$ that assigns to $(v_i,u_i)$ and $(v_1,u_j)$ a $1$ for every $i$ and for every $j$ belonging to $\{n+1, ... , m\}$, and a $0$ to the remaining vertices of the graph.
\end{proof}

\vspace{0.5cm}
We know that $\gamma_{\{R2\}}(G_{m, n})\leqslant \gamma_{2}(G_{m, n})$ for all positive integers $m$ and $n$. Moreover, this bound is sharp for $G_{2, n}$ for each $n$ and $G_{3, n}$ for $n\leqslant 13$ as well as $G_{4, 4}$.
We recall the following results of \cite{new}.
\begin{thm}\label{thm:2}
Let $n$ be a positive integer. Then the following equalities hold:
\begin{itemize}

\item[(i)]
$\gamma_{2}(G_{2, n})=n$,

\item[(ii)]
$\gamma_{2}(G_{3, n})= \lceil \frac{4n}{3}\rceil$,

\item[(iii)]
$\gamma_{2}(G_{4, n})= \lceil \frac{7n+3}{4}\rceil$, for $n\geqslant 3$.
\end{itemize}
\end{thm}

\begin{prop}
$\gamma_{\{R2\}}(G_{2, n})=n.$
\end{prop}
\begin{proof}
We claim that the weight of each layer of $P_{2}$ is at least $1$. Assume that there exists a layer with weight $0$. To have a Roman $\{2\}$-dominating set for $G_{2, n}$, the weight of the adjacent layers 
will be $4$. The obtained Roman $\{2\}$-domination number is not optimal because its weight is larger than $\gamma_2(G_{2, n})$.
\end{proof}

\begin{prop}
\qquad \qquad
\begin{enumerate}
\item[(a)]
For $n=2, 3, 6$,
$\gamma_{\{R2\}}(G_{3, n})\leqslant \lfloor \frac{5n+3}{4}\rfloor$.
Otherwise,
$\gamma_{\{R2\}}(G_{3, n})\leqslant \lceil \frac{5n+3}{4}\rceil$.

\item[(b)]
For $n=2, 3, 5, 6, 9$, $\gamma_{\{R2\}}(G_{4, n})\leqslant \lfloor \frac{5n+4}{3}\rfloor$. 
Otherwise, $\gamma_{\{R2\}}(G_{4, n})\leqslant \lceil \frac{5n+4}{3}\rceil$. 

\end{enumerate}
\end{prop}
\begin{proof}
Suppose that $v_{ij}$ is the vertex in the row $i$ and column $j$ for $1\leqslant i\leqslant m$ and $1\leqslant j\leqslant n$ in $G_{m, n}$. In each part we give a complete explanation about a basic case of the product and then we can obtain upper cases using it.  
For $(a)$, for $n=2, 3, 6$, it is sufficient to use part (ii) of Theorem \ref{thm:2}. 
Now, let $n=4k-1$ for some positive integer $k\geqslant 2$. We define a Roman $\{2\}$-dominating function $f=(V_0, V_1, V_2)$ such that $v_{ij}\in V_2$ for $j=4t$ for some positive integer $1\leqslant t\leqslant k-1$, such that $i=1$ if $t$ is odd, otherwise $i=3$.
Also, 
$$V_0=\{v_{ij}: d(v_{ij}, v)=1, 2, 4,~ \text{for some}~v\in V_2,~\text{and}~1\leqslant i\leqslant 3,~1\leqslant j\leqslant n\},$$ 
where $d(v_{ij}, v)$ is the length of shortest path between two vertices $v_{ij}$ and $v$. The label of other vertices is $1$. Hence, $w(f)=5k$. For $n\neq 4k-1$ we obtain the result by adding at most $3$ columns to the case $n=4k-1$. By adding the first column to the case $n=4k-1$, one can keep the previous assignments (in case $n=4k-1$). Set the label of $v_{2(4k)}$, $1$ and the label of the other two added vertices are zero. One can repeat this step by adding a column before the first column in case $n=4k-1$. Now, we have to check this method for adding the third column. It is easy to find a Roman $\{2\}$-dominating set if we add two vertices with label $1$. Finally, $\gamma_{\{R2\}}(G_{3, n})\leqslant \lceil \frac{5n+3}{4}\rceil$.
For $(b)$, in graphs A, B and C given in Fig. \ref{product}, a star, a black circle and a white circle denote a vertex with label $2, 1$ and $0$, respectively. We want to construct $G_{4, n}$ for $n\geqslant 7$ by merging a number of graphs A, B and C. When two of these graphs merge with each other, two of the columns turns into one column in the obtained graph. 
One can prove the first part using Theorem \ref{thm:2} and graphs in Fig. \ref{product}. 
Suppose that $n(A), n(B)$ and $n(C)$ are the number of used $A, B$ and $C$ in $G_{4, n}$, respectively. Consider $n=3k+i$ for some positive integers $k$ and $i$ such that $1\leqslant i\leqslant 3$. For $G_{4, n}$ assign $n(A)=k-i, n(C)=i-1$ and $n(B)=1$ (except for $n=9$, $n(B)=0$). Note that when $k\leqslant i$, we set $n(A)=0$. 
\end{proof}
\begin{figure}[!htb]
\minipage{0.90\textwidth}
\includegraphics[width=\linewidth]{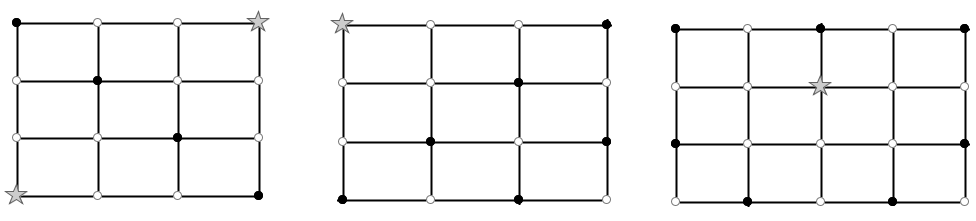}
\caption{Graphs A, B and C, respectively}\label{product}
\endminipage
\end{figure}

\vspace{2 cm}
%%%%%%%%%%%%%%%%%%%%%%%%%%%%%%%%%%%%%%%%%%

%\section*{Acknowledgment}
%We would like to thank the anonymous referee for his (her) helpful comments and
%suggestions.
%\vspace{0.3 cm}
%%%%%%%%%%%%%%%%%%%%%%%%%%%%%%%%%%%%%%%%%%%%%%%%%%%%%%%%%%%%%%%%%%%%%%%%

\vspace{-0.25cm}
%%%%%%%%%%%%%%%%%%%%%%%%%%%%%%%%%%%%%%%%%%%%%%%%%%%%%

\end{document}